\documentclass[twoside,
]{amsart}

  \usepackage[utf8]{inputenc}
  \usepackage{amsmath}
\usepackage{enumerate}
  \usepackage{amssymb}
  \usepackage{amsthm}
  \usepackage{bm}
  \usepackage{graphicx}
  \usepackage{color}
\usepackage{accents}
\usepackage{todonotes}

\newtheorem{Th}{Theorem}

\newtheorem{lem}{Lemma}
\newtheorem{rem}{Remark}

\newtheorem{Cor}{Corollary}
\newtheorem*{conj}{Conjecture}

\newtheorem*{que}{Question}
\newtheorem*{thank}{\ \ \ Acknowledgment}

\def\scalar(#1,#2){(#1\mid#2)}

\renewcommand{\hat}{\widehat}

\newcommand{\ca}{\mathcal{A}}
\newcommand{\cb}{\mathcal{B}}

\newcommand{\bmu}{\bm \mu}

\newcommand{\R}{{\mathbb{R}}}
\newcommand{\Pro}{{\mathbb{P}}}
\newcommand{\T}{{\mathbb{T}}}
\newcommand{\C}{{\mathbb{C}}}
\newcommand{\Z}{{\mathbb{Z}}}
\newcommand{\N}{{\mathbb{N}}}
\newcommand{\E}{{\mathbb{E}}}

\newcommand{\vep}{\varepsilon}

\newcommand{\mob}{\boldsymbol{\mu}}
\newcommand{\lamob}{\boldsymbol{\lambda}}
\newcommand{\bnu}{\boldsymbol{\nu}}
\newcommand{\tend}[3][]{\xrightarrow[#2\to#3]{#1}}

\newcommand{\ds}{\displaystyle}

\newcommand*{\Resize}[2]{\resizebox{#1}{!}{$#2$}}
\newcommand*{\Sup}{\Resize{0.7cm}{sup}}

\title[A cubic nonconventional ergodic average]{A cubic nonconventional ergodic average with M\"{o}bius and Liouville weight}
\author{E. H. El Abdalaoui}
\address{ Department of Mathematics, University
of Rouen, LMRS, UMR 60 85, Avenue de l'Universit\'e, BP.12, 76801
Saint Etienne du Rouvray - France}
\email{elhoucein.elabdalaoui@univ-rouen.fr }
\author{Xiangdong Ye}
\address{Wu Wen-Tsun Key Laboratory of Mathematics, USTC, Chinese Academy of Sciences, Department of Mathematics, University
of science and technology of China, Hefei, Anhui, 230026- China}
\thanks{The second author is supported by NNSF of China (11371339  and 11431012).\\
$^{1}$ In a forthcoming version, we establish that the cubic nonconventional ergodic average of any order with M\"{o}bius and Liouville weight converge almost surely to zero.}

\email{yexd@ustc.edu.cn}

\begin{document}
\date{\today}
\maketitle
{\renewcommand\abstractname{Abstract}
\begin{abstract}
It is shown that the cubic nonconventional ergodic average of order 2 with M\"{o}bius and Liouville
weight converge almost surely to zero\footnote{}. As a consequence, we obtain that the Ces\`{a}ro mean of the self-correlations
and some moving average of the self-correlations of M\"{o}bius and Liouville functions converge to zero. Our proof
gives, for any $N \geq 2$ and $\epsilon>0$,
$$\frac1{N}\sum_{m=1}^{N}\Big|\frac1{N}\sum_{n=1}^{N} \bmu(n) \bmu(n+m)\Big| \leq \frac{C}{\ln(N)^{\epsilon}},$$
and
$$\frac1{N}\sum_{m=1}^{N}\Big|\frac1{N}\sum_{n=1}^{N} \lamob(n) \lamob(n+m)\Big| \leq \frac{C}{\ln(N)^{\epsilon}}$$
where $C$ is a constant which depends only on $\varepsilon$.
\vspace{7cm}

\hspace{-0.7cm}{\em AMS Subject Classifications} (2000): 28D15 (Primary), 05D10, 11B37, 37A45 (Secondary).\\


\end{abstract}

\newpage
\section{Introduction.}
The purpose of this note is motivated, on one hand, by the recent great interest on M\"{o}bius
function from dynamical point view, and on the other hand, by the problem of the multiple recurrence which
goes back to the seminal work of Furstenberg \cite{Fur}. This later problem has nowadays a long history.\\

The dynamical studies of M\"{o}bius function was initiated recently by Sarnak in \cite{Sarnak1}\footnote{See also
\cite{Sarnak2}, \cite{Sarnak3}, \cite{Sarnak4}, \cite{Sarnak5}}. There, Sarnak made a conjecture that  M\"{o}bius
function is orthogonal to any deterministic dynamical sequence. Notice that this precise the definition of reasonable
sequence in the  M\"{o}bius randomness law mentioned by Iwaniec-Kowalski in \cite[p.338]{Iwaniec}.  Sarnak further
mentioned that Bourgain's approach allows to prove that for almost all point $x$ in any measurable dynamical system
$(X,\ca,T,\Pro)$, the M\"{o}bius function is orthogonal to any dynamical sequence $f(T^nx)$. In \cite{al-lem},
using a spectral theorem combined with Davenport estimation and Etmadi's trick, the authors gave a simple proof.
Subsequently, Cuny and Weber gave a  proof in which they mentioned that there is a rate in this almost sure
convergence \cite{Cuny-W}. They further used Bourgain's method to prove that the almost sure convergence holds
for the other arithmetical functions, like the divisor function, the theta function and  the generalized
Euler totient function. Very recently, using Green-Tao estimation \cite{Green-Tao} combined with the method in \cite{al-lem}, Eisner in \cite{Tanja}
proved that almost surely the dynamical sequence $f(T^{p(n)}x)$, where $p$ is an integer polynomial,
is orthogonal to the M\"{o}bius function. She further mentioned that the M\"{o}bius function is a good
weight (with limit 0) for the multiple polynomial mean ergodic theorem by Qing Chu's result \cite{Chu}. Subsequently, in a very recent preprint \cite{Nikos-Host},
Host and Frantzikinakis  established that any multiplicative function with mean value along any arithmetic sequence is a good weight for the multiple polynomial mean ergodic theorem (with limit 0 for aperiodic multiplicative functions). \\

Here, we are interested in the pointwise convergence of cubic nonconventional ergodic average with M\"{o}bius
and Liouville weight. The convergence of cubic nonconventional ergodic average was initiated by Bergelson in \cite{Berg},
where convergence in $L^2$ was shown for order 2 and under the extra assumption that all the
transformations are equal. Under the same assumption, Bergelson's result was extended by
Host and Kra for cubic averages of order 3  in \cite{Host-K1}, and for arbitrary order in \cite{Host-K2}.
Assani proved that pointwise convergence of cubic nonconventional ergodic average of order 3 holds for not necessarily
commuting maps in \cite{Assani1, Assani2}, and further established the pointwise convergence for cubic averages of
arbitrary order holds when all the maps are equal.
In \cite{chu-nikos}, Chu and Frantzikinakis completed the study and established the pointwise convergence for
the cubic averages of arbitrary order.\\

Very recently, Huang-Shao and Ye \cite{Ye} gave a topological-like proof of the pointwise convergence
of the cubic nonconventional ergodic average when all the maps are equal. They further applied their
method to obtain the pointwise convergence of nonconventional ergodic average for a distal system.\\

Here, we establish the pointwise convergence of cubic of order 2 with M\"{o}bius weight or Liouville weight.
The proof depends heavily on the double recurrence Bourgain's theorem (DRBT for short) \cite{BourgainD}. As a
consequence, we obtain that the C\'esaro mean of the self-correlation of M\"{o}bius function and those of
Liouville function converge to zero.\\

Using Davenport estimation, we further deduce that some moving average of the self-correlation of  M\"{o}bius function or Liouville function converge to zero, and they are summable along any divergent geometric sequence. The paper is organized as follows:\\

In section 2, we recall the main ingredients needed for the proof. In section 3, we state
our main results and its consequences. In section 4, we give a proof of our first main result when
at least one of the maps has a discrete spectrum.
In section 5, we prove our second main result which assert that the Ces\`{a}ro mean and
some moving average of the self-correlations of M\"{o}bius and Liouville functions converge to zero. In section 6, we prove our first main result when at least one of dynamical system is a nilseystem.
In section 7, we establish the Wiener-Wintner's version of Katai-Bourgain-Sarnak-Ziegler criterion,
and we end the section by proving our first main result.\\
When this paper was under preparation, we learned that Matom\"{a}ki, Radzwi{\l}{\l} and Tao \cite{MRtao} proved that for any natural number
$k$, and for any $10 \leq H \leq X$, we have
$$\sum_{1 \leq h_1,h_2,\cdots,h_k \leq H}\Big|\sum_{1 \leq n \leq X}\lamob(n+h_1)\cdots
\lamob(n+h_k)\Big|\ll k \Big(\frac{\log\log H}{\log H}+\frac1{\log^{\frac1{3000}}X}\Big) H^{k-1} X.$$
In the case $k=2$, this gives
$$\sum_{1 \leq h \leq X} \Big|\sum_{1 \leq n \leq X}\lamob(n)\lamob(n+h)\Big|\ll k \Big(\frac{\log\log H}{\log H}+\frac1{\log^{\frac1{3000}}X}\Big) H X.$$
This last estimation is largely bigger to our estimation when $H=X.$\\

We remind that besides, some estimation of limsup and liminf of some correlations of Liouville was obtain by several authors: Graham and Hansely, Harman-Pintz and Wolke, and Cassaigne-Ferenczi-Mauduit-Rivat and S\'{a}rk\"{o}zy. We referee to \cite{CF} for more details and for the references on the subject.

\section{Basic definitions and tools.}\label{Sec:dt}
Recall that the Liouville function $\lamob ~~:~~ \N^*\longrightarrow  \{-1,1\}$ is defined by
$$
\lamob(n)=(-1)^{\Omega(n)},
$$
where $\Omega(n)$ is the number of prime factors of $n$ counted with multiplicities with $\Omega(1)=1$. The integer $n$ is said to be not square-free if there is a prime number $p$ such that $n$ is in the class of $0$ mod $p^2$. The M\"{o}bius function
$\bmu ~:~ \N \longrightarrow \{-1,0,1\}$ is define as follows
$$
\bmu(n)=\begin{cases}
\lamob(n),& \text{ if $n$ is square-free ;}\\
1,& \text{ if }n=1;\\
0,& \text{ otherwise}.
\end{cases}
$$
We shall need the following crucial result due to Davenport \cite{Da} (for a recent proof see \cite{ramare}).
\begin{Th}For any $\epsilon>0$, for any $N \geq 2$, we have
\begin{eqnarray}\label{Dav1}
\sum_{n=1}^{N}\bmu(n) e^{2\pi i n  t} = O\Big(\frac{N}{\ln(N)^{\epsilon}}\Big),
\end{eqnarray}
uniformly in $t$.
\end{Th}
\noindent{}By Batman-Chowla's argument in \cite{Bat-Cho} we have\footnote{See the proof of Lemma 1 in \cite{Bat-Cho}.}
\begin{Th}
For any $\epsilon>0$, for any $N \geq 2$, we have
\begin{eqnarray}\label{Dav2}
\sum_{n=1}^{N}\lamob(n) e^{2\pi i n t} = O\Big(\frac{N}{\ln(N)^{\epsilon}}\Big),
\end{eqnarray}
uniformly in $t$.
\end{Th}
The  M\"{o}bius and Liouville functions are connected to the Riemann zeta function by the following
$$
\frac1{\zeta(s)}=\sum_{n=1}^{\infty}\frac{\mob(n)}{n^s}
\text{ and } \frac{\zeta(2s)}{\zeta(s)}=\sum_{n=1}^{\infty}\frac{\lamob(n)}{n^s}
\text{ for any }s\in\mathbb{C}\text{ with }\Re(s)>1.
$$
We remind that Chowla made a conjecture in~\cite{Cho} on the multiple self-correlation of $\bmu$ as follows:
\begin{conj}[of Chowla]
For each choice of $1\leq a_1<\dots<a_r$, $r\geq 0$, with $i_s\in \{1,2\}$, not all equal to~$2$, we have
\begin{equation}\label{cza}
\sum_{n\leq N}\mob^{i_0}(n)\cdot \mob^{i_1}(n+a_1)\cdot\ldots\cdot\mob^{i_r}(n+a_r)={\rm o}(N).
\end{equation}
\end{conj}
In \cite{Sarnak1}, Sarnak noticed  that Chowla conjecture is a notorious conjecture in number theory, and  formulated the following conjecture:
\begin{conj}[of Sarnak]
For any dynamical system $(X,T)$, where $X$ is a compact metric space and $T$ is homeomorphism of zero topological entropy, for any $f\in C(X)$ and any $x \in X$, we have
\begin{equation}\label{sar}
\sum_{n\leq N}f(T^nx)\mob(n)={\rm o}(N).
\end{equation}
\end{conj}
He further announced that Chowla conjecture implies Sarnak conjecture, and wrote ``we persist in maintaining
Conjecture \eqref{sar} as the central one even though it is much weaker than Conjecture \eqref{cza}. The point
is that Conjecture \eqref{cza} refers only to correlations of $\bmu$ with deterministic sequences and avoids the
difficulties associated with self-correlations." For the ergodic proof of the fact that Chowla conjecture implies
Sarnak conjecture we refer the reader to \cite{al-lem} \footnote{For a purely combinatorial proof see the references in \cite{al-lem} .}.\\

Given an arithmetical function $A$ ($A : \N\longrightarrow \C),$ and a positive integer $N \in \N$, for $n \in \{1,\cdots,N\}$, we define a self-correlation coefficient $c_{n,N}$ of A by

$$ c_{n,N}(A)=\frac1{N}\sum_{m=1}^{N}A(m)A(m+n).
$$
According to Wiener \cite{Wiener}, if the limit exists for each $n$, this gives the spectral measure of $A$.\\
\subsection*{Cubic averages and related topics.}

Let $(X,\cb,\Pro)$ be a Lebesgue probability space and given three measure preserving transformations $T_1,T_2,T_3$ on $X$.
Let $f_1,f_2,f_3 \in L^{\infty}(X)$.  The cubic nonconventional ergodic averages of order $2$ with weight $A$ are defined by   $$\frac1{N^2}\sum_{n,m=1}A(n)A(m)A(n+m)f_1(T_1^nx) f_2(T_2^nx) f_3(T_3^nx).$$
This nonconventional ergodic average can be seeing as a classical one as follows
$$\frac1{N^2}\sum_{n,m=1}\widetilde{f_1}({\widetilde{T_1}}^n(A,x))
{\widetilde{f_2}}({\widetilde{T_2}}^m(A,x)) {\widetilde{f_3}}({\widetilde{T_3}}^{n+m}(A,x)) ,$$
where $\widetilde{f_i}=\pi_0\otimes f_i, \widetilde{T_i}=(S \otimes T_i),~~~i=1,2,3$ and
$\pi_0$ is define by $x=(x_n)\longmapsto x_0$ on the space $Y=\C^{N}$ equipped with some probability measure.\\

The study of the cubic averages is closely and strongly related to the notion of seminorms introduced in \cite{Gowers} and \cite{Host-K2}. They are nowadays called Gowers-Host-Kra's seminorms.\\

Assume that $T$ is an ergodic measure preserving transformation on $X$. Then,  for any $k \geq 1$,
the Gowers-Host-Kra's seminorms on $L^{\infty}(X)$ are defined inductively as follows
 $$\||f|\|_1=\Big|\int f d\mu\Big|;$$
 $$\||f|\|_{k+1}^{2^{k+1}}=\lim\frac{1}{H}\sum_{l=1}^{H}\||f.f\circ T^l|\|_{k}^{2^{k}}.$$

 For each $k\geq 1$, the seminorm $\||.|\|_{k}$ is well defined. For details, we refer the reader to \cite{Host-K2} and \cite{Host-Studia}. Notice that
 the definitions of Gowers-Host-Kra's seminorms can be easily extended to non-ergodic maps as it was mentioned by Chu and
  Frantzikinakis in \cite{chu-nikos}.\\

The importance of the Gowers-Host-Kra's seminorms in the study of the nonconventional multiple ergodic averages is due to the existence of $T$-invariant sub-$\sigma$-algebra $\mathcal{Z}_{k-1}$ of $X$ that satisfies
$$\E(f|\mathcal{Z}_{k-1})=0 \Longleftrightarrow \||f|\|_{k}=0.$$
This was proved by Host and Kra in \cite{Host-K2}. The existence of the factors $\mathcal{Z}_{k}$ was also showed by Ziegler in \cite{Ziegler}.
We further notice that Host and Kra established a connection between the $\mathcal{Z}_{k}$ factors and the nilsystems in \cite{Host-K2}.\\

\subsection*{Nilsystems and nilsequences.}  The nilsystems are defined in the setting of homogeneous space \footnote{For a nice account of the theory of the homogeneous space we refer the reader to \cite{Dani},\cite[pp.815-919]{B-katok}.}. Let $G$ be a Lie
group, and $\Gamma$ a discrete cocompact subgroup (Lattice, uniform subgroup) of $G$. The homogeneous space is given by
$X=G/\Gamma$ equipped with the Haar measure $h_X$ and the canonical complete $\sigma$-algebra
$\cb_{c}$. The action of $G$ on $X$ is by the left translation, that is, for any $g \in G$,
we have $T_g(x\Gamma)=g.x\Gamma=(gx)\Gamma.$ If further $G$ is a nilpotent Lie group of order
$k$, 
$X$ is said to be a $k$-step nilmanifold. For any fixed $g \in G$, the dynamical system $(X,\cb_{c},h_X,T_g)$
is called a $k$-step nilsystem. The basic $k$-step  nilsequences on $X$ are defined by $f(g^nx\Gamma)=(f \circ T_g^n)(x\Gamma)$,
where $f$ is a continuous function of $X$. Thus, $(f(g^nx\Gamma))_{n \in \Z}$ is any element of
$\ell^{\infty}(\Z)$, the space of bounded sequences, equipped with uniform norm
$\ds \|(a_n)\|_{\infty}=\sup_{n \in \Z}|a_n|$. A $k$-step nilsequence, is a uniform limit of basic $k$-step nilsequences.
For more details on the nilsequences we refer the reader to  \cite{Host-K3}
and \cite{BHK}\footnote{The term 'nilsequence' was coined by Bergleson-Host and Kra in 2005 \cite{BHK}.}.\\

Recall that the sequence of subgroups $(G_n)$ of $G$ is a filtration if  $G_1=G,$ $G_{n+1}\subset G_{n},$ and
$[G_n,G_p] \subset G_{n+p},$ where $[G_n,G_p]$ denotes the subgroup of $G$ generated by the commutators $[x,y]=
x~y~x^{-1}y^{-1}$ with $x \in G_n$ and $y \in G_p$.\\

The lower central filtration is given by $G_1=G$ and $G_{n+1}=[G,G_n]$. It is well know that the lower central filtration allows to construct a Lie algebra $\textrm{gr}(G)$ over the ring $\Z$ of integers. $\textrm{gr}(G)$ is called a
graded Lie algebra associated to $G$ \cite[p.38]{Bourbaki2}. The filtration is said to be of degree or length $l$ if
$G_{l+1}=\{e\},$ where $e$ is the identity of $G$.\\

We denote by $G^e$ the identity component of $G$. Since $X=G/\Gamma$ is compact, we can assume that $G/G^e$ is finitely generated \cite{Leib}.\\

If $G$ is connected and simply-connected  with Lie algebra $\mathfrak{g}$ \footnote{By Lie's fundamental theorems and up to isomorphism, $\mathfrak{g}=T_eG$, where $T_eG$ is the tangent space at the identity $e$ \cite[p.34]{Kirillov}.}, then $\exp~~:~~G \longrightarrow \mathfrak{g}$ is a diffeomorphism, where $\exp$ denotes the Lie group exponential map. If $G$ is 2-step nilpotent in addition, then the multiplication law in $G$ can be expressed as follows in terms of the exponential map.
$$
\exp(X) . \exp(Y) = \exp(X +Y + \frac12 ([X, Y])) ~~~~~~~~~~~\textrm{for~~all~~~} X,Y \in  \mathfrak{g},
$$
where $[,]$ is the Lie Bracket on $\mathfrak{g}$. We further have, by Mal'cev's criterion, that $\mathfrak{g}$ admits a basis $\mathcal{X}=\{X_1,\cdots,X_m\}$ with rational structure constants \cite{Malcev}, that is,
$$[X_i,X_j]=\sum_{n=1}^{m} c_{ijn}X_n,~~~~~~\textrm{for~~all~~~}  1 \leq i,j \leq k,$$
where the constants $c_{ijn}$ are all rational. Let us precise that the Mal'cev's criterion allows us to assert that the correspondence between lattices in $G$ and rational structures in $G$ is almost one-one.\\

Let $\mathcal{X}=\{X_1,\cdots,X_m\}$ be a Mal'cev basis of  $\mathfrak{g}$, then any element $g \in G$ can be uniquely written in the form $g=\exp\Big(t_1X_1+t_2X_2+\cdots+t_m X_m\Big),$ $t_i \in \R$, since the map $\exp$ is a diffeomorphism.
The numbers $(t_1,t_2,\cdots,t_k)$ are called the Mal'cev coordinates of the first kind of $g$. In the same manner,
$g$  can be uniquely written in the form $g=\exp(s_1X_1). \exp(s_2X_2).\cdots.\exp(s_m X_m),$ $s_i \in \R$.
The numbers $(s_1,s_2,\cdots,s_k)$ are called the Mal'cev coordinates of the second kind of $g$. Applying Baker-Campbell-Hausdorff formula, it can be shown that the multiplication law in $G$ can be expressed by a polynomial mapping
$\R^m \times \R^m \longrightarrow \R^m$ \cite[p.55]{Onishchik}, \cite{Green-Tao-Orbit}. This gives that any polynomial sequence $g$ in $G$ can be written as follows
$$g(n)=\gamma_1^{p_1(n)},\cdots,\gamma_m^{p_m(n)},$$
where $\gamma_1 ,\cdots,\gamma_m \in G$, $p_i ~~:~~\N \longrightarrow \N$ are polynomials \cite{Green-Tao-Orbit}. Given $n, h \in \Z$, we put
$$\partial_{h}g(n)=g(n+h)g(n)^{-1}.$$
This can be interpreted as a discrete derivative on $G$. Given a filtration $(G_n)$ on $G$,
 a sequence of polynomial $g(n)$ is said to be adapted to $(G_n)$ if  $\partial_{h_i}\cdots \partial{h_1}g$ takes values in $G_i$ for all positive integers $i$ and for all choices of $h_1, \cdots,h_i \in \Z$. The set of all polynomial sequences adapted to $(G_n)$ is denoted by ${\textrm{poly}}(\Z,(G_n))$.\\

  Furthermore, given a Mal'cev's basis ${\mathcal{X}}$ one can induce a right-invariant metric $d_{\mathcal{X}}$ on $X$ \cite{Green-Tao-Orbit}. We remind that for a real-valued function $\phi$ on $X$, the Lipschitz norm is defined by
$$\|\phi\|_{L}=\|\phi\|_{\infty}+\sup_{x \neq y}\frac{\big|\phi(x)-\phi(y)\big|}{d_{\mathcal{X}}(x,y)}.$$
The set $\mathcal{L}(X,d_{\mathcal{X}})$ of all Lipschitz functions is a normed vector space, and for any  $\phi$ and $\psi$ in $\mathcal{L}(X,d_{\mathcal{X}})$, $\phi \psi \in \mathcal{L}(X,d_{\mathcal{X}})$ and
$\|\phi \psi \|_L \leq \|\phi\|_L \|\psi\|_L$. We thus get, by Stone-Weierstrass theorem, that the subsalgebra $\mathcal{L}(X,d_{\mathcal{X}})$ is dense in the space of continuous functions $C(X)$ equipped with uniform norm $\|\|_{\infty}$.\\ 

 It turns out that for a Lipschitz function, extension from an arbitrary subset is possible without increasing the Lipschitz norm. Thanks to Kirszbraun-Mcshane extension theorem \cite[p.146]{Dudley}.\\

In this setting, we remind the following fundamental Green-Tao's theorem on the strong orthogonality of the M\"{o}bius function to any $m$-step nilsequences, $m \ge 1$.
\begin{Th}\label{Green-Tao-th}\cite[Theorem 1.1]{Green-Tao}.
Let $G/\Gamma$ be a $m$-step  nilmanifold for some $m\ge 1$.
Let $(G_p)$ be a filtration of $G$ of degree $l \ge 1$. Suppose that $G/\Gamma$ has a $Q$-rational Mal'cev basis $\mathcal{X}$ for some $Q \ge 2$, defining
a metric $d_{\mathcal{X}}$ on $G/\Gamma$. Suppose that $F : G/\Gamma\rightarrow [-1, 1]$ is a Lipschitz function. Then, for any $A>0$, we have the bound,
$$\underset{g \in {\textrm{poly}}(\Z, (G_{p}))}{\Sup}\Big|\frac1{N}\sum_{n=1}^{N}\bmu(n)F(g(n)\Gamma)\Big| \leq C\frac{(1 + ||F||_{L})} {\log^{A} N},$$
where the constant $C$ depend on $m,l,A,Q$, $N \geq 2$.
 \end{Th}
We further need the following decomposition theorem due to Chu-Frantzikinakis and Host from \cite[Proposition 3.1]{chu-Nikos-H}.  
\begin{Th}[NSZE-decomposition theorem \cite{chu-Nikos-H}]\label{NSZE} Let $(X,\ca,\mu,T)$ be a dynamical system, $f \in  L^{\infty}(X)$,
and $k \in \N$. Then for every $\varepsilon > 0$, there exist measurable functions $f_{ns} ,f_z ,f_e $, such that
\begin{enumerate}[(a)]
\item $\|f_{\kappa}\|_{\infty} \leq 2 \|f\|_{\infty}$ with $\kappa \in \{ns,z,e\}$.
\item $f = f_{ns} + f_z + f_e$ with  $|\|f_z\|_{k+1} = 0;~~ \|f_e\|_1 <\varepsilon;$ and
\item for $\mu$ almost every $x \in X$, the sequence $(f_{ns}(T^nx))_{n \in \N}$ is a $k$-step nilsequence.
\end{enumerate}
\end{Th}

\section{Main results.}\label{sec:ms}
We state our first main result as follows.
\begin{Th}\label{main}
The cubic nonconventional ergodic average of order $2$ with M\"{o}bius weight or Liouville weight
converge almost everywhere to $0$, that is, for any $f_1,f_2,f_3 \in L^{\infty}(X),$ for almost all $x$, we have
$$ \frac1{N^2}\sum_{m,n=1}^{N}\bmu(n)\bmu(m)\bmu(n+m)f_1(T_1^nx) f_2(T_2^mx)f_3(T_3^{n+m}x)
\tend{N}{+\infty}0,$$
and
$$\frac1{N^2}\sum_{m,n=1}^{N}\lamob(n)\lamob(m)\lamob(n+m)f_1(T_1^nx) f_2(T_2^mx)f_3(T_3^{n+m}x) \tend{N}{+\infty}0.$$
\end{Th}

The strategy of the proof of our main first result can be described as follows. We first give the proof when at least one of the dynamical system has a discrete spectrum. Using Leibman's observation combined with Green-Tao's estimation, we extend our proof to the case when at least one of the dynamical system is a nilsystem. As a consequence, we get that Sarnak's conjecture holds for the class of nilsystems. Finally, we establish our first main result for the functions in the orthocomplement of $\mathcal{Z}_2$ factor.\\ 

Namely, the proof of Theorem \ref{main} is divided in three sections. In section 4, we prove the following.
\begin{Th}\label{main-2}
If for some $i \in \{1,2,3\}$, $T_i$ has a discrete spectrum. Then, the cubic nonconventional ergodic average of order $2$ with M\"{o}bius weight or Liouville
weight converge almost everywhere to $0$, that is, for any $f_1,f_2,f_3 \in L^{\infty}(X),$ for almost all $x$, we have
$$ \frac1{N^2}\sum_{m,n=1}^{N}\bmu(n)\bmu(m)\bmu(n+m)f_1(T_1^nx) f_2(T_2^mx)f_3(T_3^{n+m}x)
\tend{N}{+\infty}0,$$
and
$$\frac1{N^2}\sum_{m,n=1}^{N}\lamob(n)\lamob(m)\lamob(n+m)f_1(T_1^nx) f_2(T_2^mx)f_3(T_3^{n+m}x) \tend{N}{+\infty}0.$$
\end{Th}
In section \ref{sec:ns}, we establish the following.
\begin{Th}\label{main-3}
If for some $i \in \{1,2,3\} $, $T_i$ is a nilsystem. Then,  the cubic nonconventional ergodic average of order $2$ with M\"{o}bius weight or Liouville weight converge almost everywhere to $0$, that is,
$$ \frac1{N^2}\sum_{m,n=1}^{N}\bmu(n)\bmu(m)\bmu(n+m)f_1(T_1^nx) f_2(T_2^mx)f_3(T_3^{n+m}x)
\tend{N}{+\infty}0,$$
and
$$\frac1{N^2}\sum_{m,n=1}^{N}\lamob(n)\lamob(m)\lamob(n+m)f_1(T_1^nx) f_2(T_2^mx)f_3(T_3^{n+m}x) \tend{N}{+\infty}0.$$
\end{Th}
Finally, in  section \ref{sec:gc}, we complete the proof by proving the following.
\begin{Th}\label{main-4}
If for some $i \in \{1,2,3\}$, the function $f_i$ is in the orthocomplement of the
$\mathcal{Z}_2$ factor, i.e. if for some $i$,
then for almost all $x$, we have
$$ \frac1{N^2}\sum_{m,n=1}^{N}\bmu(n)\bmu(m)\bmu(n+m)f_1(T_1^nx) f_2(T_2^mx)f_3(T_3^{n+m}x)
\tend{N}{+\infty}0,$$
and
$$\frac1{N^2}\sum_{m,n=1}^{N}\lamob(n)\lamob(m)\lamob(n+m)f_1(T_1^nx) f_2(T_2^mx)f_3(T_3^{n+m}x) \tend{N}{+\infty}0.$$
\end{Th}
Notice that our result involves the self-correlation of $\bmu$. Furthermore, as far as we know the problem of the self-correlation of $\bmu$ is still open. This problem is known as Elliott's conjecture and it can be stated as follows.
\begin{conj}[of Elliott]\cite{Elliott-C}
 \begin{eqnarray}\label{Elliot}
  \lim_{N \longrightarrow +\infty} c_{h,N}(\bmu)=\begin{cases}
0 &{\rm {~~if~~}} h \neq 0 \\
\displaystyle \frac{6}{\pi^ 2} &{\rm {~~if~not~}}.
\end{cases}
 \end{eqnarray}
\end{conj}
P.D.T.A. Elliott wrote in his 1994's AMS Memoirs that ``even the simple particular
cases of the correlation (when $h=1$ in \eqref{Elliot}) are not well understood.
Almost surely the M\"obius function satisfies \eqref{Elliot} in this case, but at
the moment  we are unable to prove it.''\\

Recently, el Abdalaoui and Disertori in \cite{al-Diser} established that the $L^1$-flatness of
trigonometric polynomials with M\"{o}bius coefficients implies Elliott's conjecture. The $L^1$-flatness
of trigonometric polynomials with coefficients in $\{0,1,-1\}$ is an open problem in harmonic analysis
and spectral theory of dynamical systems (see \cite{al-Nad} and the references therein).\\

Nevertheless, as a consequence of Theorem \ref{main-2}, we have the following result which seems to be new.
\begin{Cor} The M\"{o}bius function and Liouville function satisfy
$$ \frac1{N^2}\sum_{n,m=1}^{N}\bmu(n)\bmu(m)\bmu(n+m)\tend{N}{+\infty}0,$$
and
$$ \frac1{N^2}\sum_{n,m=1}^{N}\lamob(n)\lamob(m)\lamob(n+m)\tend{N}{+\infty}0,$$
\end{Cor}
\begin{proof}Take $f_1=f_2=f_3=1$ in Theorem \ref{main}.
\end{proof}

Our second main result can be stated as follows.
\begin{Th}\label{Main-P} For any $\rho > 1$ there exist two sequences of integer $m_l,n_l$ such that
$$\frac{1}{[\rho^{m_l}]}\sum_{k=1}^{[\rho^{m_l}]}\bmu(k)\bmu(k+n_l)\tend{l}{+\infty} 0,$$
and
$$\frac{1}{[\rho^{m_l}]}\sum_{k=1}^{[\rho^{m_l}]}\lamob(k)\lamob(k+n_l)\tend{l}{+\infty} 0.$$ We further have, for any integer
$N \geq 2$, for any $\epsilon >0$,
$$ \frac1{N}\sum_{n=1}^{N}\Big|\frac1{N}\sum_{m=1}^{N}\bmu(m)\bmu(n+m)\Big|\leq\frac{C}{\ln(N)^{\epsilon}},$$
and
$$\frac1{N}\sum_{n=1}^{N}\Big|\frac1{N}\sum_{m=1}^{N}\lamob(m)\lamob(n+m)\Big|\leq \frac{C}{\ln(N)^{\epsilon}},$$
where $C$ is a constant which depends only on $\varepsilon.$
\end{Th}
This gives
$$\frac1{N}\sum_{n=1}^{N}\Big|\frac1{N}\sum_{m=1}^{N}\bmu(m)\bmu(n+m)\Big|\tend{N}{+\infty}0,$$
and
$$\frac1{N}\sum_{n=1}^{N}\Big|\frac1{N}\sum_{m=1}^{N}\lamob(m)\lamob(n+m)\Big|\tend{N}{+\infty}0.$$
Before starting to prove our main results, let us point out that it is suffice to establish our result for a dense set of functions. Indeed, assume that the convergence holds for some $f_1$ and let $g$ be such that $||f-g||_1<\epsilon$, for a given $\epsilon>0.$ Put
$$\psi_N(f_1,f_2,f_3)=\frac1{N^2}\sum_{m,n=1}^{N}\bmu(n)\bmu(m)\bmu(n+m)f_1(T_1^nx) f_2(T_2^mx)f_3(T_3^{n+m}x).
$$
Then,
\begin{eqnarray*}
&&|\psi_N(f_1,f_2,f_3)-\psi_N(g,f_2,f_3)|\\
 &\leq& \frac1{N}\sum_{n=1}^{N}|f_1(T_1^nx)-g(T_1^nx)|
\Big|\frac1{N}\sum_{m=1}^{N}\bmu(m)\bmu(n+m)f_2(T_2^mx)f_3(T_3^{n+m}x)\Big|\\
&\leq& ||f_2||_{\infty}||f_3||_{\infty}\frac1{N}\sum_{n=1}^{N}|f_1(T_1^nx)-g(T_1^nx)|.
\end{eqnarray*}
Letting $N$ goes to $\infty$, it follows
$$\limsup|\psi_N(f_1,f_2,f_3)-\psi_N(g,f_2,f_3)| \leq \|f_1-g\|_1 \leq \epsilon,$$
by Birkhoff ergodic theorem combined with our assumption. Notice that the maps $T_i$ are
supposed to be ergodic. From now on, without loss of generality, we will assume that $T_i,i=1,2,3$ are ergodic.

\section{On the discrete spectrum case.}\label{sec:ds}
In this section, we focus our study on the case in which at least one of the maps $T_i, i=1,2,3$ has a
discrete spectrum.

\begin{proof}[{\textbf {Proof of Theorem \ref{main-2}}}]Let assume that $T_1$ has a discrete
spectrum and let $f_1$ be an eigenfunction with eigenvalue $\lambda$. Then, for almost all  $x \in X$, we can write
\begin{eqnarray}\label{bo1}
&&\Big|\frac1{N^2}\sum_{n,m=1}^{N}\bmu(n)\bmu(m)\bmu(n+m)
f_1(T_1^nx) f_2(T_2^mx)f_3(T_3^{n+m}x)\Big|\nonumber \\
&=& \Big|\frac1{N}\sum_{m=1}^{N}\bmu(m)f_2(T_2^mx)
\frac1{N} \sum_{n=1}^{N}\bmu(n) \bmu(n+m)\lambda^n
 f_3(T_3^{n+m}x)\Big| \nonumber\\
 & \leq &  \frac1{N} \sum_{m=1}^{N} \big|f_2(T_2^mx)\big|\Big|
 \frac1{N} \sum_{n=1}^{N}\bmu(n)\bmu(n+m) \lambda^n
f_3(T_3^{n+m}x)\Big|\nonumber\\
 &\leq& \|f_2\|_{\infty} \frac1{N} \sum_{m=1}^{N} \Big|\frac1{N}\sum_{n=1}^{N} \bmu(n)\bmu(n+m) \lambda^n
f_3(T_3^{n+m}x)\Big|
\end{eqnarray}

Applying Cauchy-Schwarz inequality we can rewrite \eqref{bo1} as
\begin{eqnarray}\label{boC1}
&&\Big|\frac1{N^2}\sum_{n,m=1}^{N}\bmu(n)\bmu(m)\bmu(n+m)
f_1(T_1^nx) f_2(T_2^mx)f_3(T_3^{n+m}x)\Big|\nonumber \\
&\leq &\|f_2\|_{\infty} \Big(\frac1{N} \sum_{m=1}^{N} \Big|\frac1{N}\sum_{n=1}^{N} \bmu(n)\bmu(n+m) \lambda^n
f_3(T_3^{n+m}x)\Big|^2\Big)^{\frac1{2}}
\end{eqnarray}

Furthermore, by Bourgain's observation \cite[equations (2.5) and (2.7)]{BourgainD}, we have
\begin{eqnarray}\label{bo2}
&&\sum_{m=1}^{N}\Big|\frac1{N}\sum_{n=1}^{N} \bmu(n)\bmu(n+m) \lambda^nf_3(T_1^{n+m}x)\Big|^2 \nonumber\\
 &=& \sum_{m=1}^{N}\Big|\int_{\T}\Big(\frac1{N}\sum_{n=1}^{N}
 \bmu(n)\lambda^n z^{-n}\Big)\Big(\sum_{p=1}^{2N} \bmu(p)
 f_3(T_3^px) z^p\Big) z^{-m} dz\Big|^2 \nonumber \\
 &\leq&  \int_{\T}\Big|\frac1{N}\sum_{n=1}^{N} \bmu(n)\lambda^n
 z^{-n}\Big|^2\Big|\sum_{p=1}^{2N} \bmu(p)
 f_3(T_3^px) z^p\Big|^2 dz.
\end{eqnarray}
The last inequality is due to Parseval-Bessel inequality. Indeed, put
$$\Phi_N(z)=\Big(\frac1{N}\sum_{n=1}^{N}
 \bmu(n)\lambda^n z^{-n}\Big)\Big(\sum_{p=1}^{2N} \bmu(p)
 f_3(T_3^px) z^p\Big).$$
Then, for any $m \in \Z$,
 $$ \widehat{\Phi_N}(m)=\int_{\T}\Big(\frac1{N}\sum_{n=1}^{N}
 \bmu(n)\lambda^n z^{-n}\Big)\Big(\sum_{p=1}^{2N} \bmu(p)
 f_3(T_3^px) z^p\Big) z^{-m} dz.$$
Whence
\begin{eqnarray*}
 \sum_{m=1}^{N}\Big|\int_{\T}\Big(\frac1{N}\sum_{n=1}^{N}
 \bmu(n)\lambda^n z^{-n}\Big)\Big(\sum_{p=1}^{2N} \bmu(p)
 f_3(T_3^px) z^p\Big) z^{-m} dz\Big|^2 &=&\sum_{m=1}^{N}\Big|\widehat{\Phi_N}(m)\Big|^2\\
 &\leq& \int_{\T} |\Phi_N(z)|^2 dz.
 \end{eqnarray*}
Now, combining \eqref{boC1} with \eqref{bo2} we can assert that
\begin{eqnarray}\label{bo3}
&&\Big|\frac1{N^2}\sum_{n,m=1}^{N}\bmu(n)\bmu(m)\bmu(n+m)
f_1(T_1^nx) f_2(T_2^mx)f_3(T_3^{n+m}x)\Big|\nonumber\\&=& \Big|\frac1{N}\sum_{m=1}^{N}\bmu(m)f_2(T_2^mx)
\frac1{N} \sum_{n=1}^{N}\bmu(n) \lambda^n\bmu(n+m)
 f_3(T_3^{n+m}x)\Big|\nonumber\\
 &\leq&\|f_2\|_{\infty} \Big(\frac1{N} \sup_{z \in \T} \Big|\frac1{N}\sum_{n=1}^{N} \bmu(n)\lambda^n
 z^{-n}\Big|^2 \int_{\T}\Big|\sum_{p=1}^{2N} \bmu(p)
 f_3(T_3^px) z^p\Big|^2 dz\Big)^{\frac12}.
\end{eqnarray}
We thus get
\begin{eqnarray}\label{bo4}
&&\Big|\frac1{N^2}\sum_{n,m=1}^{N}\bmu(n)\bmu(m)\bmu(n+m)
f_1(T_1^nx) f_2(T_2^mx)f_3(T_3^{n+m}x)\Big|\nonumber\\
 &\leq&\|f_2\|_{\infty}  \sup_{z \in \T} \Big|\frac1{N}\sum_{n=1}^{N}\bmu(n)
z^{-n}\Big| . \Big(\frac1{N} \sum_{p=1}^{2N}|\bmu(p)|\Big)^{\frac12} \|f_3\|_{\infty},
\end{eqnarray}
since the map $z \mapsto \lambda z$ is a continuous bijection on the torus, and we have
$$\ds \int_{\T}\Big|\sum_{p=1}^{2N} \bmu(p) f_3(T_3^px) z^p\Big|^2 dz=\sum_{p=1}^{2N}|\bmu(p)||f_3(T_3^px)|^2
\leq \sum_{p=1}^{2N}|\bmu(p)|  \|f_3\|_{\infty}^2. $$

It follows from Davenport's estimation~\eqref{Dav1} combined with \eqref{bo4} that for each $\vep>0$, we have
\begin{eqnarray}\label{eq:Dav3}
&&\Big|\frac1{N^2}\sum_{n,m=1}^{N}\bmu(n)\bmu(m)\bmu(n+m)
f_1(T_1^nx) f_2(T_2^mx)f_3(T_3^{n+m}x)\Big|\nonumber\\
&\leq& \sqrt{2}\|f_2\|_{\infty}\|f_3\|_{\infty}\frac{C}{\ln(N)^{\varepsilon}}.
\end{eqnarray}
where $C$ is a constant which depends only on $\varepsilon$. Letting $N$ go to $\infty$, we conclude that the
almost sure convergence holds. Moreover, by  the same argument, the almost sure convergence to $0$ holds with the
Liouville weight, and the proof of the theorem is complete.
\end{proof}

\section{On the self-correlation of M\"{o}bius and Louiville functions.}\label{sec:scml}
Notice that we have actually proved
\begin{lem}For any $\varepsilon>0$, for any $N \geq 2$,
\begin{eqnarray}\label{Dav3}
\frac1{N}\sum_{n=1}^{N}\Big|\frac1{N}\sum_{m=1}^{N} \bmu(n)\bmu(n+m)\Big| \leq \frac{C}{\ln(N)^{\varepsilon}},
\end{eqnarray}
where $C$  is a constant which depends only on $\varepsilon$.
\end{lem}
\begin{proof}
  This follows by combining \eqref{bo1}, \eqref{bo2}, \eqref{bo3}
with \eqref{Dav1} and by taking $\lambda = 1$ and $f_3 = 1$.
\end{proof}
At this point, the proof of the second part of Theorem \ref{Main-P} follows.\\

Let $\rho>1$, then for $N=[\rho^m]$ with some $m\ge1$, \eqref{Dav3} takes the form

\[
 \frac1{[\rho^m]}\sum_{n=1}^{[\rho^m]}|c_{n,N}|\\ \leq \frac{C}{{(m\ln(\rho))}^{\varepsilon}}\text{,~~~~~for any~~ }\varepsilon>0.
\]
By choosing $\varepsilon=2$, we obtain
\[
\sum_{ m \geq 1}\frac1{[\rho^m]}\sum_{n\leq [\rho^m]}|c_{n,[\rho^m]}| <+\infty.
\]
Let $(\delta_l)$ be a sequence of positive numbers such that $\delta_l \tend{l}{+\infty}0.$ Then, for any $l \geq 1$
there exists a positive integer $m_l$ such that
\[
\sum_{ m \geq m_l}\frac1{[\rho^m]}\sum_{n\leq [\rho^m]}|c_{n,[\rho^m]}| <\delta_l.
\]
This gives, for any $m \geq m_l$,
\[
\frac1{[\rho^m]}\sum_{n\leq [\rho^m]}|c_{n,[\rho^m]}| <\delta_l.
\]
Hence, there exists $n_l \leq [\rho^m] $ such that
\[
|c_{n_l,[\rho^{m_l}]}| <\delta_l.
\]
By letting $l$ go to $\infty$, we get $c_{n_l,[\rho^{m_l}]} \tend{l}{+\infty} 0$. This prove the first part of Theorem~\ref{Main-P}.

Note that we have proved more, namely,
\begin{Cor}For any $\rho>1$,
$$\sum_{ m \geq 1}\frac1{[\rho^m]}\sum_{n=1}^{[\rho^m]}\Big|\frac1{[\rho^m]}
\sum_{k=1}^{[\rho^{m}]}\bmu(k)\bmu(k+n)\Big| <+\infty. $$
\end{Cor}

\begin{rem} It is shown in \cite{al-Diser} that if Sarnak's conjecture holds with some technical assumption then the self-corrections of $\bmu$ satisfy
$$\frac{1}{2N}\sum_{n=-N}^{N}\bmu(n)\bmu(n+k) \tend{N}{+\infty} 0,$$
for any $k<0$ \footnote{We extend the definition of $\bmu$ to the negative integer in usual fashion.}.
\end{rem}

\section{On the nilsystem case.}\label{sec:ns}
In this section, we present the proof of our first main result when at least one of the dynamical systems is a nilsystem.
\begin{proof}[{\textbf {Proof of Theorem \ref{main-3}}}]Let us assume that $T_1$ is an elementary nilsystem of order $s$, that is, $T_1$ is an ergodic $s$-step nilsystem on $X=G/\Gamma$, where $G$ is a nilpotent Lie group of dimension $s$ and $\Gamma$ is a discrete subgroup. By the density argument, it suffices to prove the theorem for a nilsequence $(f_1(T^nx))$, $x\in X$, $f_1$ is a continuous function on $X$. Now, by Leibman's observation \cite{Leib}, we can embed $G$ into a connected and simply-connected nilpotent Lie group $\hat{G}$ with a cocompact subgroup $\hat{\Gamma}$ such that $X=G/\Gamma$ is isomorphic to a sub-nilmanifold of $\hat{X}=\hat{G}/\hat{\Gamma},$ with all translations from $G$ represented in $\hat{G}$. Furthermore, by Tietze-Uryshon extension theorem \cite[p.48]{Dudley}, we can extend $f_1$ to $\hat{X}$. Hence, we are reduced to prove our main result for the nilsystems on $\hat{X}$. \\

Analyzing the proof given in section \ref{sec:ds}, we need to estimate
$$\sup_{z \in \T} \Big|\frac1{N}\sum_{n=1}^{N} \bmu(n) f_1(T^nx) z^{-n}\Big|.$$
Again, by the density argument, we may assume that $f_1$ is Lipschitz. We further notice that the sequence $(a_nz^{-n})_n$ can viewed as a nilsequence on
$Y=\hat{G}/\hat{\Gamma} \times \R/\T$. But the group $\hat{G}\times \R$ is connected and simply-connected, and the function
$F_1(x,z)=f_1(x)z^{-1}$ is Lipschitz. Then, we can apply Green-Tao's Theorem (Theorem 1.1 in \cite{Green-Tao}) for a given filtration $(H_n)$ of $\hat{G}\times \R$ of length $m \geq 1$. This gives,
$$\sup_{z \in \T} \Big|\frac1{N}\sum_{n=1}^{N} \bmu(n) f_1(T^nx) z^{-n}\Big| \leq C \frac{1+\|f_1\|_{L}}{\ln^{A}(N)},$$
For any $A>0$, uniformly on $x$ and $z$. Letting $N$ goes to infinity, we get
$$\sup_{z \in \T} \Big|\frac1{N}\sum_{n=1}^{N} \bmu(n) f_1(T^nx) z^{-n}\Big| \tend{N}{+\infty}0.$$
Whence,
$$\Big|\frac1{N^2}\sum_{n,m=1}^{N}\bmu(n)\bmu(m)\bmu(n+m)
f_1(T_1^nx) f_2(T_2^mx)f_3(T_3^{n+m}x)\Big| \tend{N}{+\infty}0. $$
by \eqref{bo4},  which end the proof of the theorem.
\end{proof}
Note that we have proved the following popular and well-known result.
\begin{Th} Sarnak's conjecture holds for any nilsystem.
\end{Th}
\section{On Wiener Wintner's version of Katai-Bourgain-Sarnak-Ziegler criterion.}\label{sec:kbsz}
As mentioned before, notice that the main ingredient, in the proof given in section 2, is based on the estimation of the following quantity:
$$\sup_{z \in \T}\Big|\frac1{N}\sum_{n \leq N} \bmu(n) \lambda^n z^n\Big|.$$
It follows that to tackle the general case, we need to estimate the following quantity:
\begin{eqnarray}\label{WWS}
\sup_{t}\Big|\frac1{N}\sum_{n=1}^{N}\bmu(n)f(T^nx) e^{2 \pi i n t }\Big|
\end{eqnarray}
The almost sure pointwise convergence of \eqref{WWS} without the M\"{o}bius weight follows from Wiener-Wintner theorem.
In this case, \eqref{WWS} converges to zero almost surely provided that the spectral measure of $f$ is a
continuous measure\footnote{Recall that $\sigma_f$ is a finite Borel measure on the circle determined by
its Fourier transform given by $\widehat{\sigma}_f(n)=\int f\circ T^n\cdot \overline{f}\ d\mu$, $n\in\Z$.},
that is, $f$ is in the orthocomplement of the Kronecker factor. In his 1993's paper \cite{Lesigne},
Lesigne  gave a modern proof of Wiener-Wintner theorem. The proof is based on the following
\linebreak van der Corput inequality (a proof of it can found in \cite{KN}).

\begin{Th}[van der Corput's inequality.]
Let $u_0, \cdots, u_{N-1}$ be
complex numbers, and let H be an integer with $0 \leq H \leq N-1.$ Then
$$ \Big|\frac{1}{N}\sum_{n=0}^{N-1}u_n\Big|^2
\leq$$
$$\frac{N+H}{N^2(H+1)}\sum_{n=0}^{N-1}|u_n|^2
+2\frac{N+H}{N^2(H+1)^2}\sum_{h=1}^{H}(H+1-h)~Re\Big(\sum_{n=0}^{N-h-1}u_{n+h}
\overline {u_n}\Big).$$
where $Re(z)$ denotes the real part of $z  \in \C$.
\end{Th}
If one tries to apply naively van der Corput machinery in our case, then this leads him
to the study of the self-correlation of order 2 of $\bmu$ which seems to be an outstanding problem as
previously mentioned in section \ref{sec:ms}. This is due, as pointed by Sarnak in \cite{Sarnak4},
to the lack of methods for the study of the self-correlation of $\bmu$.
Nevertheless, there are methods to apply for the study of the correlation of $\bmu$ with certain
sequences. Thanks to the bilinear method of Vinogradov \cite{Vi}. This allows Bourgain-Sarnak-Ziegler
to produce a criterion as a tool in the study of Sarnak's conjecture. Here, we state a
Wiener-Wintner's version of it ((WWKBSZ for short): Wiener-Wintner's version of K\'{a}tai-Bourgain-Sarnak-Ziegler
criterion \footnote{K\'{a}tai's version of the criterion is stated in \cite{Katai}.}).

\begin{Th}[WWKBSZ criterion]\label{WWKBSZ}
Let $(X,\ca,\mu)$ be a Lebesgue probability space and $T$ be an invertible measure preserving transformation. Let $\bnu$ be a multiplicative function\footnote{An arithmetical function $\bnu$ is said to be multiplicative if
$\bnu(mn)=\bnu(m)\bnu(n)$ whenever $n$ and $m$ are coprime.}. Let $f$ be in $L^{\infty}$ with $\|f\|_{\infty} \leq 1$ and
$\varepsilon>0$ and assume that for almost all point $x\in X$ and for all different prime numbers $p$ and $q$ less than $\exp(1/\varepsilon)$, we have
  \begin{equation}
    \label{eq:f_k limit}
     \limsup_{N\to\infty}\sup_{t}\left| \dfrac{1}{N}\sum_{n=1}^N e^{2 \pi i n (p-q) t}f (T^{pn}x) f(T^{qn}x) \right| < \varepsilon,
  \end{equation}
  then, for almost all $x \in X$, we have
  \begin{equation}
    \label{eq:almost orthogonality for f_k}
     \limsup_{N\to\infty} \sup_{t}\left|\dfrac{1}{N} \sum_{n=1}^N \bnu(n) e^{2 \pi i n t}f(T^{n}x) \right| < 2\sqrt{\varepsilon\log1/\varepsilon}.
  \end{equation}
\end{Th}
\begin{proof}The proof is, indeed word-for-word the same as that of Theorem 2 in \cite{BSZ},
except that at the equation (2.7) one need to apply the following elementary inequality:
for any two bounded positive functions $F$ and $G$, we have
$$\sup(F(x)+G(x)) \leq \sup(F(x))+\sup(G(x)).$$
\end{proof}
It follows from the WWKBSZ criterion that we need to estimate \eqref{eq:f_k limit}. But this follows by a careful application of van der Corput trick combined with Bourgain's double recurrence theorem (BDRT for short) as it was shown by Assani, Duncan and Moore in \cite{Assani-Moore}. Precisely, they proved the following
\begin{Th}[WW's version of BDRT \cite{Assani-Moore}]\label{WWBDRT}Let $(X,\ca,\mu)$ be a Lebesgue probability space and $T$ be an invertible measure preserving transformation. Let $f,g \in L^{\infty}(X)$. Then, for any $k,l \in \Z$ with $|\max(k,l)|>0$, for almost all $x \in X$, we have
$$\E\big(f|_{\mathcal{Z}_2}\big)=0 \Longrightarrow \sup_{t}\Big|\dfrac{1}{N}\sum_{n=1}^{N}e^{2 \pi i n t}f(T^{kn}x)g(T^{ln}x)\Big|\tend{N}{+\infty}0.$$
\end{Th}

\section{Proof of the main results-Theorem \ref{main-4} and Theorem \ref{main}.}\label{sec:gc}
Without loss of generality we assume that $\E(f_1|_{\mathcal{Z}_2})=0.$ We start by rewriting the equation \eqref{bo3} and \eqref{eq:Dav3} in the following form
 \begin{eqnarray}\label{bo-cs3}
&&\Big|\frac1{N^2}\sum_{n,m=1}^{N}\bmu(n)\bmu(m)\bmu(n+m)
f_1(T_1^nx) f_2(T_2^mx)f_3(T_3^{n+m}x)\Big|\nonumber\\
 &\leq&\|f_2\|_{\infty}  \sup_{z \in \T} \Big|\frac1{N}\sum_{n=1}^{N} \bmu(n)f_1(T_1^nx)
 z^{-n}\Big| \Big(\frac1{N}\int_{\T}\Big|\sum_{p=1}^{2N} \bmu(p)
 f_3(T_3^px) z^p\Big|^2 dz\Big)^{\frac12},
\end{eqnarray}
and
\begin{eqnarray}\label{eq:bo-cs4}
&&\Big|\frac1{N^2}\sum_{n,m=1}^{N}\bmu(n)\bmu(m)\bmu(n+m)
f_1(T_1^nx) f_2(T_2^mx)f_3(T_3^{n+m}x)\Big|\nonumber\\
&\leq& \sqrt{2}\|f_2\|_{\infty}\|f_3\|_{\infty} \sup_{z \in \T}  \Big|\frac1{N}\sum_{n=1}^{N} \bmu(n)f_1(T_1^nx)
 z^{-n}\Big|.
\end{eqnarray}

\noindent{}Since
$$\E\big(f_{1}|_{\mathcal{Z}_2}\big)=0,$$
it follows that,  for any $k,l \in \Z$ with $|\max(k,l)|>0$, for almost all $x \in X$, we have
$$ \sup_{t}\Big|\dfrac{1}{N}\sum_{n=1}^{N}e^{2 \pi i n t}f_{1}(T_1^{kn}x)f_{1}(T_1^{ln}x)\Big|\tend{N}{+\infty}0,$$
by WW's version of BDRT (Theorem \ref{WWBDRT}).
Whence
\[
\sup_{t}\Big|\frac1{N}\sum_{n=1}^{N}\bmu(n)f_{1}(T_1^nx) e^{2 \pi i n t }\Big| \tend{N}{+\infty}0,
\]
by WWKBSZ criterion (Theorem \ref{WWKBSZ}). This combined with \eqref{eq:bo-cs4} gives
$$\Big|\frac1{N^2}\sum_{n,m=1}^{N}\bmu(n)\bmu(m)\bmu(n+m)
f_{1,z}(T_1^nx) f_2(T_2^mx)f_3(T_3^{n+m}x)\Big| \tend{N}{+\infty}0,$$
and the proof of Theorem \ref{main-4} is complete.
\hfill{$\Box$}

\begin{rem}Notice that our result is valid if the weight is given by any multiplicative function $\bnu$ bounded by 1 and for which the following condition is satisfy
\begin{eqnarray}\label{DD}
\underset{g \in {\textrm{poly}}(\Z, (G_{p}))}{\Sup}\Big|\frac1{N}\sum_{n=1}^{N}\bnu(n)F(g(n)\Gamma)\Big|\tend{N}{+\infty}0,
\end{eqnarray}
where $(G_p)$ is a given filtration in a nilpotent group $G$,  $F : G/\Gamma\rightarrow [-1, 1]$ is a Lipschitz function and
$G/\Gamma$ is equipped with the metric $d_{\mathcal{X}}.$ We refer to condition \eqref{DD} as the strong Daboussi-Delange's condition.
\end{rem}
We end this section by giving the proof of our main result.

\begin{proof}[\textbf{Proof of Theorem \ref{main}}.]Let $f_1 \in L^{\infty}$. Then, by  NSZE-decomposition theorem (Theorem \ref{NSZE}), we can decompose $f_1$ for $k=2$ as follows
$$ f_1=f_{1,sc}+f_{1,z}+f_{1,e}.$$
By Theorem \ref{main-3} combined with Theorem \ref{main-4}, we have
$$\Big|\frac1{N^2}\sum_{n,m=1}^{N}\bmu(n)\bmu(m)\bmu(n+m)
f_{1,sc}(T_1^nx) f_2(T_2^mx)f_3(T_3^{n+m}x)\Big| \tend{N}{+\infty}0,$$
and
$$\Big|\frac1{N^2}\sum_{n,m=1}^{N}\bmu(n)\bmu(m)\bmu(n+m)
f_{1,z}(T_1^nx) f_2(T_2^mx)f_3(T_3^{n+m}x)\Big| \tend{N}{+\infty}0.$$
Therefore, for almost all $x \in X$, we have
\begin{eqnarray*}
&&\limsup \Big|\frac1{N^2}\sum_{n,m=1}^{N}\bmu(n)\bmu(m)\bmu(n+m)
f_{1}(T_1^nx) f_2(T_2^mx)f_3(T_3^{n+m}x)\Big|\\
&=& \limsup \Big|\frac1{N^2}\sum_{n,m=1}^{N}\bmu(n)\bmu(m)\bmu(n+m)
f_{1,e}(T_1^nx) f_2(T_2^mx)f_3(T_3^{n+m}x)\Big|\\
&\leq& \|f_2\|_{\infty}\|f_3\|_{\infty}\limsup \frac1{N}\sum_{n=1}^{N}|f_{1,e}(T^nx)|.
\end{eqnarray*}
But, by Birkhoff theorem, for almost all $x$, we have
$$\limsup \frac1{N}\sum_{n=1}^{N}|f_{1,e}(T^nx)|=\|f_{1,e}\|_1 \leq \epsilon.$$
Whence
$$\limsup \Big|\frac1{N^2}\sum_{n,m=1}^{N}\bmu(n)\bmu(m)\bmu(n+m)
f_{1}(T_1^nx) f_2(T_2^mx)f_3(T_3^{n+m}x)\Big| \leq \|f_2\|_{\infty}\|f_3\|_{\infty} \epsilon. $$
Since $\epsilon$ was arbitrary, we conclude that
$$\limsup \Big|\frac1{N^2}\sum_{n,m=1}^{N}\bmu(n)\bmu(m)\bmu(n+m)
f_{1}(T_1^nx) f_2(T_2^mx)f_3(T_3^{n+m}x)\Big|=0,$$
which achieves the proof of our main result.
\end{proof}
\begin{que}A natural problem suggested by our result is the following: do we have, for any $k \geq 3$,
\[
\dfrac1{N^k}\sum_{{\boldsymbol{n}} \in [1,N]^k}\prod_{{\boldsymbol{e}} \in C^*}\bnu({\boldsymbol{n}}.{\boldsymbol{e}})f_{{\boldsymbol{e}}}
\big(T_{{\boldsymbol{e}}}^{{\boldsymbol{n}}.{\boldsymbol{e}}}x\big) \tend{N}{+\infty}0?
\]
where ${\boldsymbol{n}}=(n_1,\cdots,n_k)$, ${\boldsymbol{e}}=(e_1,\cdots,e_k)$, $C^*=\{0,1\}^k\setminus\{(0,\cdots,0)\}$,
${\boldsymbol{n}}.{\boldsymbol{e}}$ is the usual inner product, and $\bnu$ is a bounded multiplicative function which satisfy
a strong Daboussi-Delange condition.\\

A second question is related to Sarnak's conjecture. Assume that $T$ satisfy Sarnak's conjecture,
 do we have for any continuous function, for all $x \in X$,
\[
\sup_{t}\Big|\frac1{N}\sum_{n=1}^{N}\bmu(n)f(T^nx) e^{2 \pi i n t }\Big| \tend{N}{+\infty}0? \textrm{~~(WWS)~~~}
\]
 Notice that the topological entropy of the cartesian product of two dynamical flow on compact set is the sum of their topological entropy \cite{GoodW}.\\
\end{que}

\begin{thank} The first author would like to express his heartfelt thanks to Professor Benjamin Weiss
for the discussions on the subject. It is a great pleasure also for him to acknowledge the warm
hospitality of University of Science and Technology of China and Fudan University where a part of this work has been done.
\end{thank}

\end{document}